\documentclass[11pt,reqno]{amsart}
\usepackage{amsmath,amsfonts,amssymb,amscd,amsthm,amsbsy,epsf}
\newtheorem{thm}{Theorem}

\newtheorem*{Lem}{Lemma}
\def\F{{\mathcal F}}
\def\real{{\mathbb R}}
\begin{document}
\title[Convolution estimates and the Gross-Pitaevskii hierarchy]
{Convolution estimates and the Gross-Pitaevskii hierarchy}
\author{William Beckner}
\address{Department of Mathematics, The University of Texas at Austin,
1 University Station C1200, Austin TX 78712-0257 USA}
\email{beckner@math.utexas.edu}
\begin{abstract}
Extensions to higher-dimensions are given for a convolution estimate used by Klainerman
and Machedon in their study of uniqueness of solutions for the Gross-Pitaevskii hierarchy.
Such estimates determine more general forms of Stein-Weiss integrals involving restriction
to smooth submanifolds. 
\end{abstract}
\maketitle

Analysis of the Gross-Pitaevskii hierarchy has led to the development and 
application of functional analytic mappings  for the rigorous description of many-body
interactions in quantum dynamics. 
In their formative and influential paper on uniqueness of solutions for 
the Gross-Pitaevskii hierarchy, Klainerman and Machedon determine 
uniform bounds for a three-dimensional convolution integral.
The idea of their argument rests on an extension of the classical convolution 
for Riesz potentials
$$\int_S \frac1{|w-g|^\lambda}\ \frac1{|y|^\mu}\, d\sigma$$
where $S$ is a smooth submanifold in $\real^n$, $w \in \real^m$ 
and the objective is to bound the size of the integral by an inverse power of 
$|w|$ under suitable conditions on $\lambda$ and $\mu$.
Such an estimate can be viewed as a step in the larger and dual program for 
understanding how smoothness controls restriction to  a non-linear sub-variety 
(see \cite{Beckner-MRL}). 
Two natural extensions to higher dimensions are suggested here:
\begin{gather}
|w|^2 \int_{\real^n\times\cdots\times\real^n} 
\delta \Big[ \tau + \mathop{{\sum}'} |x_k|^2 - |x_n|^2\Big] 
\delta \Big(w-\sum x_k\Big) \prod |x_k|^{-(n-1)} 
dx_1\cdots dx_n \label{eq:natexten1}\\
\noalign{\vskip6pt}
|w|^{n-1} \int_{\real^n\times\real^n\times\real^n} 
\delta \left[ \tau + |z|^2 + |x|^2 - |y|^2\right] 
\delta (w-x-y-z) 
\big[ |z|\, |x|\, |y|\big]^{-(n-1)} dx\,dy\,dz \label{eq:natexten2}
\end{gather}
with the objective being to determine uniform bounds in terms of the variables $\tau>0$
and $w\in \real^n$ with $n>1$ (here the prime on the symbol for sum,  product 
or sequence  
indicates that the last term should be dropped). 
From  the dilation character of the expression, one can use ``uniformity'' to 
eliminate one variable so it suffices to consider $\tau = 1$:
\begin{gather*}
\Lambda_n (w) = |w|^2 \int_{\real^n\times\cdots\times\real^n} 
\delta \Big[ 1+\mathop{{\sum}'}   |x_k|^2 - |x_n|^2\Big] 
\delta \Big(w-\sum x_k\Big) \prod |x_k|^{-(n-1)} dx_1\cdots dx_n\\
\noalign{\vskip6pt}
\Delta_n (w) = |w|^{n-1} \int_{\real^n\times\real^n\times\real^n} 
\delta \left[ 1+|z|^2 + |x|^2 - |y|^2\right] 
\delta (w-x-y-z) \big[ |x|\, |y|\, |z|\big]^{-(n-1)} dx\, dy\, dz
\end{gather*}
One observes that the first expression is an extension of the classical 
convolution form 
$$(g* f_1 * \cdots * f_n) (w)\ ,\quad 
g\in L^1 (\real^n)\ ,\quad 
f_k \in L^{n/(n-1)} (\real^n)$$
which is uniformly continuous  and in the class $C_0 (\real^n)$ using the 
Riemann-Lebesgue lemma.
Here the convolution for Lebesgue classes is replaced by Riesz potentials,
but the multivariable integration is constrained to be on a hyperbolic 
surface invariant under action by the indefinite orthogonal group.

\begin{thm}\label{thm-Lambda}
$\Lambda_n(w)$ is bounded for $n\ge 3$;
$\Delta_n (w)$ is bounded for $n\ge 2$.
\end{thm}

The argument for the proof of   Theorem~\ref{thm-Lambda} will be developed 
in several steps and will be reduced to the second statement when the 
dimension is at least four.
Note that $\Lambda_3 = \Delta_3$, and this is the case determined by Klainerman
and Machedon.

\begin{proof} 
{\sc Step 1:} for $n=2$, $\Lambda_2 (w)$ is unbounded. 
This case is instructive and will identify the method used later in the proof of the
second part. 
\begin{align*}
\Lambda_2 (w) & = |w|^2 \int_{\real^2} \delta \left( 1+|w-y|^2 - |y|^2\right) 
\frac1{|w-y|}\ \frac1{|y|}\,dy\\
\noalign{\vskip6pt}
& = |w|^2 \int_1^\infty \int_{-\pi/2}^{\pi/2} \delta 
\left[ 1+|w|^2 - 2rw\cos \theta\right] 
\frac1{\sqrt{r^2-1}}\, dr\,d\theta
\end{align*}
(since $\cos\theta$ must be positive and $|y| >1$)
\begin{align*}
& = 2|w|^2 \int_0^1 \frac1{\sqrt{1-u^2}}\ 
\frac1{\sqrt{(1+|w|^2)^2 - 4|w|^2 u^2}}\, du\\
\noalign{\vskip6pt}
& = \frac{|w|^2}{1+|w|^2} \int_0^1 \frac1{\sqrt{u}}\ \frac1{\sqrt{1-u}}\ 
\frac1{\sqrt{1-\beta u}}\, du\ ,\qquad 
\beta = \frac{4|w|^2}{(1+|w|^2)^2} \le 1\\
\noalign{\vskip6pt}
& = \pi \frac{|w|^2}{1+|w|^2}  \ F\left(\frac12,\frac12; 1;\beta\right) 
=  \frac{2|w|^2}{1+|w|^2}\ K(\sqrt{\beta}\,)
\end{align*}
where $F$ denotes the hypergeometric function and $K$ the complete elliptic 
integral. 
$\Lambda_2 (w) = \infty$ for any  $w$ on the unit sphere $|w|=1$.
Observe that for $\beta \simeq 1$ (e.g., $|w|\simeq1$) 
$$\Lambda_2 (w) \simeq  -  \ln \left( \sqrt{1-\beta}/2\right)$$

A similar calculation will now give:

\begin{Lem}
For $n=2$ and $\frac12 <\alpha <1$
$$\Lambda_{2,\alpha} (w) = |w|^{2\alpha} \int_{\real^2} \delta 
\left( 1+|w-y|^2 - |y|^2\right) |w-y|^{-\alpha} |y|^{-\alpha}\,dy$$
is uniformly bounded in $w$.
\end{Lem}

\noindent 
Observe that by dilation symmetry this result is equivalent to uniform 
boundedness with $\tau >0$, $w\in \real^2$ for 
$$|w|^{2\alpha} \int_{\real^2}\delta \left( \tau + |w-y|^2 - |y|^2\right) 
|w-y|^{-\alpha} |y|^{-\alpha}\,dy\ .$$
\medskip

\noindent {\sc Step 2:}
let $n\ge 4$; then using the second delta function for the variable $x_{n-2}$
\begin{align*}
\Lambda_n (w) 
& = |w|^2 \int_{\real^n\times\cdots\times\real^n} 
\delta \bigg[ 1+\sum^{n-3} |x_k|^2 + \Big|  w- \sum^{n-3} x_k 
- x_{n-1} - x_n\Big|^2 + |x_{n-1}|^2 - |x_n|^2 \bigg]\ \times \\
\noalign{\vskip6pt}
& \prod_{k=1}^{n-3} |x_k|^{-(n-1)} \Big|  w- \sum^{n-3}
x_k - x_{n-1} - x_n\Big|^{-(n-1)} 
\Big( |x_{n-1}|\,    |x_n|\Big)^{-(n-1)} 
dx_1 \cdots dx_{n-3} \,dx_{n-1} \,dx_n\\
\noalign{\vskip6pt}
& = |w|^2 \int_{\underbrace{\real^n\times\cdots\times \real^n}_{(n-3)\text{ copies}}}
\prod^{n-3} |x_k|^{-(n-1)} \Big| w-\sum^{n-3} x_k\Big|^{-(n-1)}\ \times \\
\noalign{\vskip6pt}
&\left[ \Big| w-\sum^{n-3} x_k\Big|^{n-1} 
\int_{\real^n\times \real^n} 
\delta \Big[ 1+\sum |x_k|^2 + \Big| w-\sum x_k-x-y\Big|^2 
+ |x|^2 - |y|^2\Big] \  \right.\times  \\
\noalign{\vskip6pt}
&\qquad  \left[ \Big| w-\sum x_k -x-y\Big|  |x|\, |y|\right]^{-(n-1)} dx\,dy \biggr]
dx_1\ldots dx_{n-3}\\
\noalign{\vskip6pt}
&\le c_n |w|^2 \int_{\real^n\times\cdots\times \real^n} 
\prod^{n-3} |x_k|^{-(n-1)} \Big| w-\sum^{n-3} x_k\Big|^{-(n-1)} 
dx_1 \ldots dx_{n-3}
\end{align*}
where 
$$c_n = \sup_{\tau,v} |v|^{n-1}\int_{\real^n\times \real^n} 
\delta \Big[ \tau + |v-x-y|^2 + |x|^2 - |y|^2\Big]
\Big[ |v-x-y|\, |x|\, |y|\Big]^{-(n-1)} dx\,dy 
= \sup_w \Delta_n (w)$$
where in the earlier expression, $\tau = 1+\sum^{n-3} |x_k|^2$ and 
$v = w- \sum^{n-3} x_k$. 
Then 
$$\Lambda_n (w) \le c_n |w|^2 
\int_{\real^n\times \cdots \times \real^n} 
\prod |x_k|^{-(n-1)} \Big| w-\sum x_k\Big|^{-(n-1)} 
dx_1 \ldots dx_{n-3}$$
Using the following notation for the Fourier transform and its action on Riesz 
potentials
\begin{align*}
(\F f)(x) 
& = \int_{\real^n} e^{2\pi ixy} f(y)\,dy \\
\noalign{\vskip6pt}
\F \Big[ |x|^{-\lambda}\Big](\xi) 
& = \pi^{-n/2 + \lambda} \frac{\Gamma (\frac{n-\lambda}2)}{\Gamma (\frac{\lambda}2)}
\ |\xi|^{-(n-\lambda)}
\end{align*}
\begin{align*}
& |w|^2  \int_{\real^n\times \cdots\times \real^n} 
\prod^{n-3} |x_k|^{-(n-1)} \Big| w- \sum^{n-3} x_k\Big|^{-(n-1)} dx_1 \cdots dx_{n-3}\\
\noalign{\vskip6pt}
&\qquad 
= \pi^{[(n-1)^2 -3]/2} \left[ \Gamma \Big(\frac{n-1}2\Big)\right]^{-(n-2)} 
\left[ \Gamma \Big(\frac{n}2 -1\Big)\right]^{-1}
\end{align*}
Hence $\Lambda_n (w)$ is bounded for $n\ge 3$ if $\Delta_n(w)$ is bounded
for $n\ge 3$.
\renewcommand{\qed}{}
\end{proof}

\noindent
An intriguing feature of this argument is that $(n-1)$ is the unique uniform inverse 
power where one can preserve dilation invariance and obtain a reduction of this 
type that connects bounds for integrals of the form $\Lambda_n$, $\Delta_n$.
Perhaps this circumstance reflects a larger underlying symmetry in addition to 
the correspondence with the property that the convolution of $n$ functions in 
$L^{n/(n-1)} (\real^n)$ will be uniformly continuous. 
\medskip

\noindent {\sc Step 3:} consider $n=2$
\begin{align*}
\Delta_2 (w) 
& = |w| \int_{\real^2\times\real^2} 
\delta \Big[  1+ |w-x|^2 + |x-y|^2 - |y|^2\Big]
\Big[ |w-x|\, |x-y|\, |y|\Big]^{-1} \,dx\, dy\\
\noalign{\vskip6pt}
& = |w| \int_{\real^2} dx \int_{\real^2} dy\,
\delta \Big[ 1+ |w-x|^2 + |x|^2 - 2|x|\, |y|\, \cos \theta\Big]  
|w-x|^{-1} |y|^{-1} 
\Big[ |y|^2 -1-|w-x|^2\Big]^{-1/2}\\
\noalign{\vskip6pt}
& = 2|w| \int_{\real^2} dx \, |w-x|^{-1} \int_0^1 \frac1{\sqrt{1-u^2}}
\Big[ (1+|w-x|^2 + |x|^2)^2 - 4 |x|^2 u^2 (1+|w-x|^2)\Big]^{-1/2} du\\
\noalign{\vskip6pt}
& = |w| \int_{\real^2} |w-x|^{-1} 
\bigg[ \Big[ 1+|w-x|^2 + |x|^2\Big]^{-1} \int_0^1 
\frac1{\sqrt u}\ \frac1{\sqrt{1-u}}\ \frac1{\sqrt{1-\beta(x)u}} \,du\bigg] dx
\end{align*}
where 
$$\beta (x) = \frac{4|x|^2 (1+|w-x|^2)}{(1+|w-x|^2 + |x|^2)^2} \le 1\ .$$
Since for $0<u<1$, $\sqrt{1-u} \sqrt{1-\beta (x)} \le 1-u\beta (x)$
$$\Delta_2 (w) \le \biggl[ \int_0^1 u^{-1/2} (1-u)^{-3/4} du\bigg] |w| 
\int_{\real^2} |w-x|^{-1} \Big[ 1+|w-x|^2 + |x|^2\Big]^{-1/2} 
\Big| 1+|w-x|^2 - |x|^2\Big|^{-1/2}\, dx $$
Set $w = |w|\xi$, dilate by $|w|$ and choose $\xi$ as the $x_1$ direction
$$\Delta_2 (w) \le c \int_{\real^2} \left(  |x_1 -1|^2 + |x_2|^2\right)^{-1/2} 
\left[ 1+4 \Big[ (x_1 -1/2)^2 + |x_2|^2\Big]\right]^{-1/2} 
\Big|x_1 - \frac12 \Big( 1+\frac1{|w|^2}\Big)\Big|^{-1/2}\, 
dx_1 \,dx_2$$
Rearrange in the variable $x_1$ using 
$$\int_{\real^2} f(x,y) g(x,y) h(x,y)\, dx\,dy 
\le \int_{\real^2} f_\# (x,y) g_\# (x,y) h_\# (x,y)\,dx\,dy$$
where $f_\# (x,y)$ is the equimeasurable symmetric decreasing  rearrangement 
of $|f(x,y)|$ in the variable $x\in \real$. 
Then 
$$\Delta_2 (w) \le c\int_{\real^2} |x|^{-1} (1+4|x|^2)^{-1/2} |x_1|^{-1/2}\,dx$$
which is a convergent integral as one sees by using polar coordinates. 
Hence $\Delta_2 (w)$ is uniformly bounded.
The option to directly use rearrangement depends on the choice of the 
inverse power, e.g., the value $(n-1)$. 
\medskip

\noindent{\sc Step 4:} 
by using simple radial coordinate estimates, one can obtain for $n>2$ 
($c$ denotes a generic constant) 
$$\Delta_n(w) \le c\Delta_2 (\bar w)$$
where $\bar w \in \real^2$ with $|w| = |\bar w|$.

Observe that the integrands for both expressions treated here, $\Lambda_n(w)$ 
and $\Delta_n(w)$, are functions only of lengths and polar angles so that 
facilitates the simplicity of the argument.
\begin{align*}
\Delta_n (w) & = |w|^{n-1} \int_{\real^n\times\real^n} 
\delta \Big[ 1+ |w-x|^ 2+ |x|^2 - 2x\cdot y\Big] 
|w-x|^{-n-1} |y|^{-n-1}\ \times\\
\noalign{\vskip6pt}
&\hskip2in 
\Big[ |y|^2 -1 - |w-x|^2\Big]^{-(n-1)} \,dx\,dy\\
\noalign{\vskip6pt}
& = \frac{\sigma (S^{n-2)}}2 |w|^{n-1} \int_{\real^n} |w-x|^{-(n-1)} 
(2|x|)^{n-2} \left( 1+ |w-x|^2 + |x|^2\right)^{-(n-1)}\ \times \\
&\hskip2in
\int_0^1 u^{-1/2} (1-u)^{(n-3)/2} \left( 1-\beta (x)u\right)^{-(n-1)/2}\,du\,dx
\end{align*}
with $\beta (x)$ as before. 
Since $0\le\beta (x)\le 1$ and $0\le u\le 1$
$$(1-u)^{(n-3)/2} (1-\beta (x) u)^{-(n-1)/2} 
\le (1-u)^{-1/2} (1-\beta (x) u)^{-1/2}$$
In the integral over $\real^n$. first make the change of variables $z= w-x$, 
and then dilate $z$ by $|w|$ and integrate out the non-polar angle variables.
\begin{gather*}
\Delta_n (w)   \le 2^{n-4} \left[ \sigma (S^{n-2})\right]^2 |w|^2 
\int_0^\infty \int_0^\pi 
\left[ \frac{|w|^2 |z-\xi |\, |z|}{1+w[ |z-\xi |^2 + |z|^2]}\right]^{n-2}\ \times \\
\noalign{\vskip6pt}
\left[ 1+ |w|^2 \left( |z-\xi|^2 + |z|^2\right)\right]^{-1} 
d|z| (\sin \theta )^{n-2}\,d\theta 
\int_0^1 u^{-1/2} (1-u)^{-1/2} \left( 1-\beta (|w|(z-\xi))u\right)^{-1/2} du
\end{gather*}
Using 
$$\left[ \frac{|w|^2 |z-\xi |\, |z|}{1+ |w|^2 (|z-\xi |^2 + |z|^2)}\right] 
\le \frac12\ ,$$
\begin{align*}
\Delta_n (w) &\le \frac18 \left[ \sigma (S^{n-2})\right]^2 |w|^2 
\int_0^\infty \int_{-\pi}^\pi \frac1{|z|} 
\left[ 1+|w|^2 (|z-\xi|^2 + |z|^2)\right]^{-1} \\
\noalign{\vskip6pt}
&\qquad \int_0^1 u^{-1/2} (1-u)^{-1/2}
\left( 1-\beta (|w|(z-\xi))u\right)^{-1/2} \,du \, 
|z| d|z| d\theta
\end{align*}
Now since we can take $\xi$ as defining the polar angle for  the 
coordinate system, and $|z-\xi|$ only depends on this angle and the length
$|z|$, $z$ and $\xi$ can be repositioned as vectors in $\real^2$ so that 
\begin{align*}
\Delta_n(w) & \le \frac18 \left[ \sigma (S^{n-2})\right]^2 |w|^2 
\int_{\real^2} \frac1{|z|} \left[ 1 +|w|^2 \big( |z-\xi|^2 + |z|^2\big)\right]^{-1}\\
\noalign{\vskip6pt}
&\qquad \int_0^1 u^{-1/2} (1-u)^{-1/2} 
\left( 1-\beta \big( |w| (z-\xi)\big) u\right)^{-1/2} \,du\,dz
\end{align*}
Reversing the previous coordinate changes of dilation and translation
\begin{align*}
\Delta_n(w) & \le \frac18 \left[\sigma (S^{n-2})\right]^2 |w| 
\int_{\real^2} |w-x|^{-1} \left[ 1+ |x|^2 + |w-x|^2\right]^{-1/2}\\
\noalign{\vskip6pt}
&\qquad 
\int_0^1 u^{-1/2} (1-u)^{-1/2} \left( 1-\beta (x)u\right)^{-1/2} \,du\, dx
\end{align*}
which then gives the required control 
$$\sup_{w\in \real^n} \Delta_n(w) 
\le \frac18 \left[ \sigma (S^{n-2})\right]^2 
\sup_{w\in \real^2}  \Delta_2 (w)$$
and hence the uniform bound for $\Delta_2 (w)$ gives a uniform bound 
for $\Delta_n (w)$, $n>2$.
This completes the proof of Theorem~\ref{thm-Lambda}. 

As noted above, the inverse power $|x|^{-(n-1)}$ has a special role for the 
convolution estimates discussed here; still the two-dimensional result from 
the Lemma is suggestive that useful bounds might be obtained for 
inverse powers close to $\alpha = n-1$.
Consider for $\tau >0$ and $w\in \real^n$
\begin{equation}\label{eq-inversepower}
|w|^\rho \int_{\real^n\times\cdots\times\real^n}
\delta \left[ \tau + \mathop{{\sum}^\prime} 
%\Sigma' 
|x_k|^2 - |x_n|^2\right]
\delta \left[ w-\sum x_k\right] \Pi |x_k|^{-\alpha} \,dx_1 \ldots dx_n
\end{equation}
For dilation invariance, $p = 2+n (\alpha -n+1)$ with the further requirement 
of positivity for the possibility of boundedness; that means 
$(n-1) \ge \alpha > (n-1) -2/n$ so that asymptotically $a\simeq n-1$ for 
large dimension. 
The upper bound is required by the nature of the proof for uniform boundedness.
As with Theorem~\ref{thm-Lambda}, the proof for uniform bounds will 
depend on a reduced integral form:
\begin{equation}\label{eq-reduced-integral}
|w|^\sigma \int_{\real^n\times\real^n} \delta \left[ \tau +|x|^2 - |y|^2\right] 
\delta (w-x-y) |x|^{-\alpha} |y|^{-\alpha} \,dx\,dy
\end{equation}
Here $\sigma = 2\alpha +2-n$ for dilation invariance.
For both forms, it suffices to show uniform bounds for $\tau=1$, and in two 
dimensions they are the same and already proved in the argument for the Lemma.

\begin{thm}\label{thm-uniform-bd}
For $n\ge 2$, $\sigma = 2\alpha + 2-n$ and $(n-1)/2 < \alpha < (n-1)$ 
\begin{equation*}
\Theta_{n,\alpha}  (w) 
= |w|^\sigma \int_{\real^n\times\real^n} 
\delta \left[ 1 + |x|^2 - |y|^2\right]
\delta (w-x-y) |x|^{-\alpha} |y|^{-\alpha} \,dx\,dy 
\end{equation*}
is uniformly bounded for $w\in \real^n$.
\end{thm}

\begin{proof}
Let $n\ge 2$: 
\begin{align*}
\Theta_{n,\alpha} (w) 
& = |w|^\sigma \int_{\real^n} \delta \left( 1+ |w-y|^2 - |y|^2\right) |w-y|^{-\alpha} 
|y|^{-\alpha} \,dy \\
\noalign{\vskip6pt}
& = \frac{2\pi^{(n-1)/2}}{\Gamma ((n-1)/2)} |w|^\sigma 
\int_1^\infty\!\! \int_0^1 \delta \left( 1+|w|^2 - 2|w| ru\right) (r^2-1)^{-\alpha/2} 
r^{n-\alpha-1} (1-u^2)^{(n-3)/2} \,dr\,du\\
\noalign{\vskip6pt}
&= \frac{\pi^{(n-1)/2} 2^{2\alpha-n}}{\Gamma ((n-1)/2)}
\left[ \frac{|w|^2}{1+|w|^2}\right]^{2\alpha -n+1} \int_0^1 (1-u)^{(n-3)/2}
u^{\alpha -(n-1)/2\ -1} \big( 1-\beta (2)u\big)^{-\alpha/w} \,du
\end{align*}
where $\beta (w) = \frac{4|w|^2}{(1+|w|^2)^2} \le 1$ 
with $(1-\beta (w)u)^{-\alpha/2} \le (1-u)^{-\alpha/2}$.
Then 
\begin{align*}
\Theta_{n,\alpha} (w) 
& \le \frac{\pi^{(n-1)/2} 2^{2\alpha-n}}{\Gamma ((n-1)/2)} 
\left[ \frac{|w|^2}{1+|w|^2}\right]^{2\alpha-n+1} 
\int_0^1 (1-u)^{(n-1-\alpha)/2\ -1} u^{\alpha-(n-1)/2\ -1} \,du\\
\noalign{\vskip6pt}
& = \pi^{(n-1)/2} 2^{2\alpha-n}      
\ \frac{\Gamma ((n-1-\alpha)/2)\ \Gamma ((2\alpha -n+1)/2)}
{\Gamma ((n-1)/2)\ \Gamma (\alpha/2)}
\ \left[ \frac{|w|^2}{1+|w|^2}\right]^{2\alpha -n+1}
\end{align*}
Hence $\Theta_{n,\alpha} (w)$ is bounded for $(n-1)/2 <\alpha <n-1$.
Note that $\Lambda_{2,\alpha} = \Theta_{2,\alpha}$.
\renewcommand{\qed}{}
\end{proof}

\begin{thm}\label{thm3}
For $n\ge 3$, $\rho = 2+n(\alpha-n+1)$ and $n-1 >\alpha > n-1-2/n$ 
\begin{equation*}
\Lambda_{n,\alpha} (w) 
= |w|^\rho \int_{\real^n\times\cdots \times\real^n} 
\delta \Big[ 1+\sum\nolimits^\prime |x_k|^2 - |x_n|^2\Big] 
\delta \Big( w-\sum x_k\Big) \Pi |x_k|^{-\alpha} \,dx_1 \ldots dx_n
\end{equation*}
is uniformly bounded for $w\in \real^n$.
\end{thm}

\begin{proof}
The argument here rests on the boundedness of $\Theta_{n,\alpha}$ following the 
method of Step~2 in the proof of Theorem~\ref{thm-Lambda}.
For $n\ge 3$ use the second delta function for the variable $x_{n-1}$ and write 
\begin{align*}
\Lambda_{n,\alpha} (w) & = |w|^\rho 
\int_{\underbrace{\scriptstyle \real^n\times\cdots\times\real^n}_{(n-2)\text{ copies}}}
\mathop{\Pi}\limits^{n-2}  |x_k|^{-\alpha} \big| w-\sum^{n-2} x_k\big|^{-\sigma}\ \times\\
\noalign{\vskip6pt}
& \bigg[ \big| w-\sum^{n-2} x_k\big|^\sigma 
\int_{\real^n} \delta \Big[ 1 + \sum |x_k|^2 + \big| w-\sum x_k-y\big|^2 - |y|^2\Big] \ \times\\
\noalign{\vskip6pt}
&\hskip.5in
 \Big[ \big| w-\sum x_k -y\big|\ |y|\Big]^{-\alpha}\,dy \bigg] \, dx_1 \ldots dx_{n-2}\\
\noalign{\vskip6pt}
&\le c_n |w|^\rho \int_{\real^n\times\cdots\times\real^n} 
\mathop{\Pi}\limits^{n-2} |x_k|^{-\alpha} \big| w-\sum^{n-2} x_k\big|^{-\sigma} 
dx_1 \ldots dx_{n-2}
\end{align*}
where 
\begin{equation*}
c_n = \sup_{\tau,v} |v|^\sigma \int_{\real^n} 
\delta \left[ \tau + |v-y|^2 - |y|^2\right] \, \big[ |v-g|\, |y|\big]^{-\alpha} \, dy
= \sup_w \Theta_{n,\alpha} (w)
\end{equation*}
and in the earlier expression, 
$\tau = 1 +\sum^{n-2} |x_k|^2$ and $v= w- \sum^{n-2} x_k$.
Then 
\begin{equation*}  
\Lambda_{n,\alpha} (w) \le c_n |w|^\rho \int_{\real^n\times\cdots\times\real^n} 
\Pi |x_k|^{-\alpha} \big| w-\sum x_k\big|^{-\sigma} dx_1 \ldots dx_{n-2}
\end{equation*}
with the integral 
\begin{equation*}
|w|^\rho \int_{\real^n\times \cdots\times \real^n} \Pi |x_k|^{-\alpha} 
\big|  w-\sum x_k\big|^{-\sigma} \, dx_1 \ldots dx_{n-2}
\end{equation*}
being constant in $w$ so that $\Lambda_{n,\alpha} (w)$ is bounded for $n\ge 3$
if $\Theta_{n,\alpha} (w)$ is bounded for $n\ge 3$.
\renewcommand{\qed}{}
\end{proof}

In surveying the estimates outlined above, the critical computation would seem 
to be the surface integral 
\begin{equation*}
\int_S \frac1{|w-y|^\lambda} \ \frac1{|y|^\mu}\, dv
\end{equation*}
which then the convolution algebra for Riesz potentials allows an extended 
multilinear result. 
For completeness, an outline is given for non-uniform Riesz potentials.

\begin{thm}\label{thm4}
For $n\ge 2$, $\sigma = \alpha +\lambda + 2-n$, $\alpha +\lambda > n-1$ 
and $0<\alpha < n-1$
\begin{equation}
\Theta_{n,\alpha,\lambda} (w) 
= |w|^\sigma \int_{\real^n\times\real^n} 
\delta \left[ 1+ |x|^2 - |y|^2\right] \delta (w-x-y) |x|^{-\alpha} |y|^{-\lambda} \,dx\,dy
\end{equation}
is uniformly bounded for $w\in \real^n$.
\end{thm}

\begin{proof}
Let $n\ge 2$; observe that since $|y| \ge1$, there is no upper bound for $\lambda$ 
in this computation.
\begin{align*}
\Theta_{n,\alpha,\lambda} (w) 
& = |w|^\sigma \int_\real^n
\delta \left[ 1+  |w-y|^ 2- |y|^2\right] |w-y|^{-\alpha} |y|^{-\lambda}\,dy\\
\noalign{\vskip6pt}
& = \frac{2\pi^{(n-1)/2}} {\Gamma ((n-1)/2)} |w|^\sigma 
\int_1^\infty \int_0^1 \delta \left( 1+ |w|^2 -2|w| ru\right) (r^2 -1)^{-\alpha/2} 
r^{n-\lambda-1}  (1+u^2)^{(n-3)/2} \,dr\,du\\
\noalign{\vskip6pt}
& = \frac{2^{\alpha+\lambda-n} \pi^{(n-1)/2}} {\Gamma ((n-1)/2)} 
\left[ \frac{|w|^2} {1+|w|^2} \right]^{\alpha + \lambda -n+1} 
\int_0^1 (1-u)^{(n-3)/2} u^{(\alpha +\lambda -n+1)/2\ -1} 
(1-\beta (w)u)^{-\alpha/2}\,du 
\end{align*}
where $\beta (w) = \frac{4|w|^2}{(1+|w|^2)^2} \le 1$ with 
$(1-\beta (w) u)^{-\alpha/2} \le (1-u)^{-\alpha/2}$.
Then 
\begin{align*}
\Theta_{n,\alpha,\lambda} (w)
& = \frac{2^{\alpha+\lambda-n} \pi^{(n-1)/2}} {\Gamma ((n-1)/2)} 
\left[ \frac{|w|^2}{1+ |w|^2}\right]^{\alpha +\lambda -n+1} 
\int_0^1 u^{(\alpha +\lambda -n+1)/2\ -1} 
(1-u)^{(n-1-\alpha)/2\ -1} \,du\\
\noalign{\vskip6pt}
& = 2^{\alpha +\lambda -n}  \pi^{(n-1)/2} 
\frac{\Gamma ((\alpha +\lambda -n+1)/2)\ \Gamma ((n-1-\alpha)/2)}
{\Gamma ((n-1)/2)\ \Gamma (\lambda/2)} 
\left[ \frac{|w|^2}{1+|w|^2} \right]^{\alpha +\lambda -n+1}
\end{align*}
Hence $\Theta_{n,\alpha,\lambda} (w)$ is bounded for $\lambda >0$, 
$0<\alpha <n-1$ and $\alpha +\lambda > n-1$.
\end{proof}

\begin{thm}\label{thm5}
For $n\ge 3$, consider real-valued exponents $0<\alpha_k < n$, $k= 1,\ldots,n-1$
and $\lambda >0$ with $\alpha = \sum \alpha_k$ and $\rho = 2+\alpha + \lambda
-n (n-1)$ so that $0<\rho <n$.
Further assume one exponent $\alpha_i$ together with $\lambda$ satisfies:
$0<\alpha_i < n-1$ and $n-1 < \alpha_i + \lambda < 2(n-1)$; relabel this $a_i$ as
$\alpha_{n-1}$.
Then 
$$\Lambda_{n,\alpha,\lambda} 
= |w|^\rho \int_{\real^n\times\real^n} 
\delta \bigg[ 1 + \mathop{{\sum}'} |x_k|^2 - |x_n|^2\bigg] 
\delta \Big(w-\sum x_k\Big)   \Pi |x_k|^{-\alpha_k}  |x_n|^{-\lambda} \,
dx_1 \ldots dx_n$$
is uniformly bounded for $w\in \real^n$.
\end{thm}

\begin{proof} 
As in Theorem~\ref{thm3}, the argument here rests on the uniform boundedness of 
$$|w|^\sigma \int_{\real^n\times\real^n} 
\delta \left[ \tau + |x|^2 - |y|^2\right] 
\delta (w-x-y) |x|^{-\alpha_{n-1}} |y|^{-\lambda} \,dx\,dy$$
for $\tau>0$ and $w\in \real^n$ which is determined by Theorem~\ref{thm4}. 
For $n\ge 3$ use the second delta function for the variable $x_{n-1}$ and 
write with $\sigma = \alpha_{n-1} + \lambda +2 -n$
\begin{align*}
\Lambda_{n,\alpha,\lambda} 
& = |w|^\rho 
\int_{\underbrace{\scriptstyle \real^n\times\cdots\times\real^n}_{(n-2) \text{ copies}} }
 \mathop{\Pi}\limits^{n-2}
 |x_k|^{-\alpha_k} \Big| w-\sum^{n-2} x_k\Big|^{-\sigma}\ \times \\
 \noalign{\vskip6pt}
&\qquad \bigg[ \Big| w-\sum^{n-2} x_k\Big|^\sigma \int_{\real^n} 
\delta\bigg[ 1+\sum |x_k|^2 + \Big| w- \sum x_k -y\Big|^2 - |y|^2 \bigg] \ \times\\
\noalign{\vskip6pt}
&\qquad \qquad 
\Big| w-\sum x_k -y\Big|^{-\alpha_{n-1}} |y|^{-\lambda}\bigg] \,dx_1\ldots dx_{n-2} \\
\noalign{\vskip6pt}
& \le c_{n,\alpha,\lambda} |w|^\rho \int_{\real^n\times\cdots \times\real^n} 
 \mathop{\Pi}\limits^{n-2}  
 |x_k|^{-\alpha_k} \Big| w-\sum^{n-2} x_k\Big|^{-\sigma} 
 dx_1 \ldots dx_{n-2}
\end{align*}
where 
$$c_{n,\alpha,\lambda} 
= \sup_{\tau,v} |v|^\sigma \int_{\real^n} 
\delta \left[ \tau + |v-y|^2 - |y|^2\right] |v-y|^{-\alpha_{n-1}} |y|^{-\lambda}$$
and in the earlier expression $\tau = 1+ \sum^{n-2} |x_k|^2$ and 
$v= w- \sum^{n-2} x_k$.
Then 
$$\Lambda_{n,\alpha,\lambda} (w) 
\le c_{n,\alpha,\lambda} |w|^\rho 
\int_{\real^n\times\cdots\times\real^n} \Pi |x_k|^{-\alpha_k} 
\Big| w-\sum x_k\Big|^{-\sigma} dx_1\ldots dx_{n-2}$$
with the integral 
$$|w|^\rho \int_{\real^n\times\cdots\times\real^n} \Pi |x_k|^{-\alpha_k} 
\Big| w-\sum x_k\Big|^{-\sigma} dx_1 \ldots dx_{n-2}$$
being constant in $w$ so that $\Lambda_{n,\alpha,\lambda} (w)$ is bounded 
in $w$ for $n\ge 3$ if $\Theta_{n,\alpha,\lambda} (w)$ is bounded for $n\ge 3$
subject to the conditions on $\alpha_{n-1}$ and $\lambda$.
\end{proof}

Implicit in the formulation of the problems treated here is the continuing 
development of new forms that characterize control by smoothness for size.
As an example and a consequence of the principal estimate obtained here, 
bounds for   new Stein-Weiss integrals with a kernel determined by restriction 
to a smooth submanifold can be shown.

\begin{thm}\label{thm6}
Define
\begin{align*}
K(w,v) &= \int_{\real^n\times\cdots\times\real^n} 
\Pi |x_k|^{-(n-1)}  \left[ \big| w - \sum x_k\big|\, \big| v-\sum x_k\big|\right]^{-(n-1)}\ \times\\
\noalign{\vskip6pt}
&\hskip1in \delta \left[ 1+ \mathop{{\sum}'} |x_k|^2 - |x_n|^2\right] 
\,dx_1 \ldots dx_n\ ,\quad n\ge 3\ ;\\
\noalign{\vskip6pt}
K_{n,\alpha} (w,v) & = \int_{\real^n\times\cdots\times \real^n}
\Pi |x_k|^{-\alpha} \left[ \big|w-\sum x_k\big|\, \big|v-\sum x_k\big|\right]^{-\lambda}\ \times\\
\noalign{\vskip6pt}
&\hskip1in \delta \left[ 1+ \mathop{{\sum}'} |x_k|^2 - |x_n|^2\right] \,dx_1 \ldots dx_n\ ,
\quad n\ge 3\\
\noalign{\vskip6pt}
&\lambda = n(n-\alpha +1)/2\ -\ 1\ ,\quad n-1 > \alpha > n-2\ -\ 2/n\\
\noalign{\vskip6pt}
H_{n,\alpha} (w,v) & = \int_{\real^n\times\real^n}  
|x|^{-\alpha} |y|^{-\alpha} 
\big[ |w-x-y|\,|v-x-y|\big]^{-(3n/2\ -\, 1-\alpha)} \ \times\\
\noalign{\vskip6pt}
&\hskip1in \delta \left( 1+ |x|^2 - |y|^2\right)\,dx\,dy\ ,\quad n\ge 2\ ,\ 
(n-1)/2 < \alpha < n-1\ ;\\
\noalign{\vskip6pt}
J(w,v) & = \int_{\real^n\times\real^n\times\real^n} 
\big[ |x|\, |y|\, |z|\big]^{-(n-1)} 
\bigl |w-x-y-z|\, |v-x-y-z|\big]^{-(n+1)/2}\ \times \\
\noalign{\vskip6pt}
&\hskip1in 
\delta \left[ 1+|x|^2 + |z|^2 - |y|^2\right]\,dx\,dy\,dz\ ,\quad n\ge 2\ ;
\end{align*}
then for non-negative $f\in L^2 (\real^n)$  and $T$ representing the above kernels 
\begin{equation}\label{eq5}
\int_{\real^n\times\real^n} f(w) T(w,v) f(v)\,dw\,dv 
\le c\int_{\real^n} |f|^2\,dx
\end{equation}
\end{thm}

\begin{proof}
Apply Pitt's inequality and the uniform bounds obtained above for 
$\Lambda_n$, $\Lambda_{n,\alpha}$, $\Theta_{n,\alpha}$ and $\Delta_n$.
Here $c$ is a generic constant.
\end{proof}

Practical application for such convolution-type estimates 
has proved to be efficient by replacing the Riesz potentials with
the Fourier transform of  
Bessel potentials (\cite{CP}, \cite{Kirk});		
advantage is achieved by removing local singularities while gaining
integrability on the potential side and improving the range of application 
as ``smoothing operators";
still the lack of homogeneity limits determination of precise dependence 
on parameters in computing best size estimates. 
But as with exact model calculations, the role of Riesz potentials can result in ``very
elegant and useful formulae" that underline intrinsic geometric structure, 
capture essential features
of symmetry and uncertainty, and provide insight to precise lower-order effects.

%%%%%%%%%%%%%%%%%%%%%

\end{document}